\documentclass[12pt]{article}
\usepackage{bbm, fancyhdr}
\usepackage{amssymb}
\usepackage{mathrsfs}
\usepackage{mathtools}
\usepackage{amsmath}
\usepackage{float}
\usepackage{amsthm}
\usepackage[utf8]{inputenc}
\usepackage[english]{babel}
\usepackage{comment}
\usepackage{listings}
\usepackage{enumitem}
\usepackage{scalerel}
\usepackage{stackengine,wasysym}
\usepackage{titlefoot}
\usepackage{marvosym}
\usepackage[backref=page]{hyperref}
\usepackage{xcolor}
\definecolor{darkgreen}{RGB}{0,125,0}
\hypersetup{colorlinks,
	breaklinks,
	urlcolor=magenta,
	linkcolor=magenta,
	citecolor=darkgreen}
\renewcommand*{\backref}[1]{}
\renewcommand*{\backrefalt}[4]{%
	\ifcase #1 (Not cited.)%
	\or        (Cited on page~#2.)%
	\else      (Cited on pages~#2.)%
	\fi}
\usepackage[capitalise,nameinlink]{cleveref}
\usepackage{aliascnt}
\usepackage{todonotes}
\usepackage{tikz-cd}

\newtheorem{theorem}{Theorem}[section] 

\newaliascnt{proposition}{theorem}
\newtheorem{proposition}[proposition]{Proposition}
\aliascntresetthe{proposition}
\crefname{proposition}{Proposition}{Propositions}

\newaliascnt{proposition_definition}{theorem}

\aliascntresetthe{proposition_definition}
\crefname{proposition_definition}{Proposition/Definition}{Propositions/Definitions}

\newaliascnt{corollary}{theorem}

\aliascntresetthe{corollary}
\crefname{corollary}{Corollary}{Corollaries}

\newaliascnt{lemma}{theorem}
\newtheorem{lemma}[lemma]{Lemma}
\aliascntresetthe{lemma}
\crefname{lemma}{Lemma}{Lemmas}

\theoremstyle{definition}

\newaliascnt{remark}{theorem}
\newtheorem{remark}[remark]{Remark}
\aliascntresetthe{remark}
\crefname{remark}{Remark}{Remarks}

\newaliascnt{claim}{theorem}

\aliascntresetthe{claim}
\crefname{claim}{Claim}{Claims}

\newaliascnt{definition}{theorem}
\newtheorem{definition}[definition]{Definition}
\aliascntresetthe{definition}
\crefname{definition}{Definition}{Definitions}

\newcounter{proofstepcounter}[theorem]
\newcounter{proofstep}

\AddToHook{env/proof/begin}{\setcounter{proofstep}{0}}

\NewDocumentCommand{\newProofStep}{ o }{%
	\par\medskip\noindent
	\refstepcounter{proofstep}%
	\IfNoValueTF{#1}{%
		\textbf{Step \theproofstep.}
	}{%
		\textbf{Step \theproofstep. #1.}
	}%
	\quad\ignorespaces
}

\crefname{proofstep}{Step}{Steps}
\Crefname{proofstep}{Step}{Steps}

\newcommand\N{\mathbb{N}}
\newcommand\Z{\mathbb{Z}}

\newcommand\R{\mathbb{R}}
\newcommand\C{\mathbb{C}}
\newcommand\T{\mathbb{T}}

\newcommand\eval{\textrm{ev}}
\newcommand\GL{\textrm{GL}}

\newcommand\ord{\textrm{ord}}
\newcommand\lcm{\textrm{lcm}}

\newcommand\End{\textrm{End}}

\newcommand\LieAlg[1]{\mathfrak{#1}}
\newcommand\Spec{\mathrm{Spec}}

\newcommand\spanSimple{\mathrm{span}}

\newcommand\supp{{\textrm{supp}}}

\newcommand\cdef[1]{{\bf \textrm{#1}}}
\newcommand\zarClosure[1]{{\overline{\langle #1 \rangle}^{\textrm{Zar}}}}
\newcommand\holo{\mathcal{O}}
\newcommand\stab{{\textrm{Stab}}}
\newcommand\Hom{\textrm{Hom}}

\newcommand\Aut{\textrm{Aut}}
\newcommand\alg{\textrm{alg}}
\newcommand\Tot{\textrm{Tot}}

\newcommand{\quasireg}{\textrm{qr}}

\renewcommand\Im{\textrm{Im}}

\newcommand\identity{\rm{Id}}
\renewcommand\Re{\textrm{Re}}
\newcommand\Iso{\textrm{Iso}}

\newcommand\codim{\textrm{codim}}

\newcommand\Fr{\textrm{Fr}}

\newcommand\Sing{\textrm{Sing}}

\renewcommand{\leftmark}{{\scriptsize Holomorphic tensors on products of algebraic cones}}

\begin{document}
	\begin{center}
	{\Large\bf  Holomorphic tensors on products of algebraic cones}\\[5mm]
	{\large Vlad Marchidanu\footnote{Partly supported by the PNRR-III-C9-2023-I8 grant CF 149/31.07.2023 {\em Conformal Aspects of Geometry and Dynamics}.\\[1mm]
		
		
		}}\\[4mm]
	\end{center}
	
	{\small
		\hspace{0.15\linewidth}
		\begin{minipage}[t]{0.7\linewidth}
			{\bf Abstract.} We study the product $C$ of two algebraic cones equipped with algebraic structures given by contractions. First we show that any holomorphic tensor on a quotient of $C$ by a group containing a contraction on both factors is invariant under the Zariski closure of this contraction when the factors have dimension $\geq 2$. We then give an explicit embedding of the cone of a Sasaki manifold to a normal variety. Using it and the result on algebraic cones, we prove that any holomorphic tensor on the product of two Sasaki manifolds is invariant under the flows of the Reeb fields.
		\end{minipage}
	}

\tableofcontents
 
\section{Introduction}

An algebraic cone can be defined abstractly as an algebraic variety over $\C$ with a single singular point to which it admits a contraction (\cref{def::alg_cone_contraction}). 
In \cite{ornea_verbitsky::algebraic_cones}, the authors endow algebraic cones with an algebraic structure (\cref{def::algebraicStructure}) defined in terms of the contraction, and show that the definition is independent of the choice of contraction (\cite[Theorem~2.9]{ornea_verbitsky::algebraic_cones}). To do this, they enrich a mesh of ideas exposed primarily in \cite{ornea_verbitsky::convexShells} and \cite{ornea_verbitsky::LeeClassesLCKPotential}, which establishes the equivalence between algebraic cones and Stein completions of $\Z$-covers of locally conformally K\"ahler (LCK) manifolds with proper potential (\cref{thm::algebraicConesAreZCoversLCK}). Thus, even though the arguments do not involve their differential geometry, LCK manifolds appear as a natural geometric framework and motivation for the study of algebraic cones.

Locally conformally K\"ahler manifolds are an expansive generalisation of K\"ahler ones.
The more restrictive subclass of LCK manifolds with potential was introduced in \cite{ornea_verbitsky::2010LCKwithPotential} and proved a rich topic of study and a useful tool ever since (see \cite{ornea_verbitsky::LCKRank}, \cite{ornea_verbitsky::topology}, \cite{ornea_verbitsky::convexShells}, etc., or the monograph \cite{lck_book}). Sasaki manifolds, on the other hand, are often seen as an odd-dimensional counterpart to K\"ahler ones (via \cref{def::Sasaki}). They are also intimately related to Vaisman -- thus LCK -- geometry, since their metric cone produces Vaisman manifolds once quotiented by a $\Z$-action (\cite{ornea_verbitsky::structureTheoremVaisman}). 

In \cite{ornea_verbitsky::holomorphicTensors}, the authors prove that any holomorphic tensor on a Vaisman manifold is invariant under the flows of the Lee and anti-Lee field (\cite[Theorem~5.1]{ornea_verbitsky::holomorphicTensors}), a result known by \cite{tsukada::holomorphicForms} for holomorphic vector fields and differential forms. The starting motivation of the present paper is to obtain a similar result for products of Sasaki manifolds studied in \cite{marchidanu::prodSasaki}, where we interpreted in a different way a family of complex structures on it and showed that this family cannot bear LCK metrics. 

In part, this goal places us naturally in the more general context of (products of) algebraic cones, which we treat in \cref{sec::algebraic_cones}, \cref{sec::products_algebraic_cones}, and \cref{sec::holom_tensors_prod_alg_cones}. We obtain (\cref{thm::holomorphicTensorsProdConesFromBase}) that any holomorphic section of a tensor bundle on the quotient of a product of algebraic cones by a $\C$-action containing a contraction on both factors is invariant under the Zariski closure of the contraction (with respect to an algebraic structure, obtained in \cref{thm::algebraicStructureFromContraction}).

We apply this to the product of Sasaki manifolds in \cref{thm::holomorphicTensorsProdSasaki} to obtain that any holomorphic tensor on the product of Sasaki manifolds is invariant under the flows of the Reeb fields. 
The passage to this context, however, is difficult and serves as a pretext to compile several results from algebraic geometry and complex analysis. 
Namely, we require an explicit embedding of (the product of) the cone(s) over Sasaki manifolds into $\C^K \setminus \{ 0 \}$ such that its closure is a normal variety.

For this, we begin by revisiting the strategy of \cite{vanCoevering::2011ExamplesAsymptoticallyConical}; as the author notes, the first proof of an embedding result is given in \cite{ornea_verbitsky::embeddingsCompactSasaki}, but it is not explicit. Meanwhile, the literature using sections of orbibundles is primarily concerned only with embedding the orbifold's underlying analytic space (\cite{ross_thomas::weightedProjectiveEmbeddings}), which does not suffice for our purposes.

We therefore give a proof of this embedding, not just for the purpose of self-containment. There are more important aspects that we require and which, to our knowledge, the existing literature doesn't bridge together, thus requiring careful assembly. Firstly, the finite generation as a $\C$-algebra of the ring used for embedding requires the orbibundle viewpoint, not developed in \cite{vanCoevering::2011ExamplesAsymptoticallyConical}; we treat it in \cref{proposition::sectionRingFinitelyGenerated}. Secondly, to show that the image under the embedding is the (product of) Stein completion(s), we also need to prove algebraic normality of a ring of sections (\cref{step::RIntegrallyClosed} of \cref{thm::holomorphicTensorsProdSasaki}), for which the embedding technique is leveraged explicitly. Finally, explicit control over the embedding is needed to allow for a particular kind of generators (see \cref{step::weightDecomposition} of \cref{thm::holomorphicTensorsProdSasaki}). For such generators, invariance under Zariski closure reduces to easy computations and density arguments.

The paper is organized as follows. In \cref{sec::algebraic_cones} we give the precise definition of algebraic cones and recall the results necessary to define their Stein completions; we recall also the definition of LCK manifolds with (proper) potential and the equivalence between closed algebraic cones and covers of Stein completions of LCK manifolds with proper potential (\cref{thm::algebraicConesAreZCoversLCK}). In \cref{sec::products_algebraic_cones} we set up the framerwork for \cref{sec::holom_tensors_prod_alg_cones} and show in \cref{thm::algebraicStructureFromContraction} how to endow the product of two closed algebraic cones with an algebraic structure (\cref{def::algebraicStructure}), while in \cref{sec::holom_tensors_prod_alg_cones} we display the mechanism which guarantees invariance under the Zariski closure based on invariance under a generator in \cref{lemma::sectionsCoherentSheavesProductConesAreInvariant}, used in \cref{thm::holomorphicTensorsProdConesFromBase}. In \cref{sec::sasakianManifolds} we define Sasaki manifolds and set up a preferred orbifold atlas for a quasiregular vector field, as well as the line bundle we will work with onwards, while in section \cref{sec::sectionRingAmpleness} we consider the ring of sections of powers of this line bundle, show how to identify it with homogeneous functions on the Sasakian cone (\cref{lemma::sectionRingsAreFunctions}), and show that it is finitely generated (\cref{proposition::sectionRingFinitelyGenerated}). In \cref{sec::holomTensorsProdSasaki} we briefly recall some aspects about complex structures on the product of Sasaki manifolds, then prove the main result \cref{thm::holomorphicTensorsProdSasaki}, relying on \cref{thm::holomorphicTensorsProdConesFromBase} and \cref{sec::sasakianManifolds} and \cref{sec::sectionRingAmpleness}.

\vspace{2em}

\textit{Acknowledgements.}
I am very grateful to Misha Verbitsky for suggesting this topic, for the very inspirational and insightful discussions we've had, and for his warm hospitality at IMPA. I am also highly thankful to Liviu Ornea for his very useful remarks, guidance, and constant support.

\section{Algebraic cones}
\label{sec::algebraic_cones}

	In this section we review the literature on algebraic cones, 
	presenting several equivalent ways to define them 
	and considering some of their properties. We recall first:
	
	\begin{definition}
		\label{def::contraction}
		Let $X$ be a complex variety and $x \in X$. A biholomorphism $\gamma \in \Aut(M)$ is called a \cdef{holomorphic contraction to $x$} if:
		\begin{itemize}
			\item $\gamma(x) = x$ and
			\item for any open set $U$ with $x \in U$ and any $K \subset X$ compact, there exists $N \gg 0$ such that $\gamma^N(K) \subset U$.
		\end{itemize}
	\end{definition}
	
	Complex varieties endowed with contractions are too general to make an object of study by available methods. Therefore, we restrict ourselves to Stein varieties:
	
	\begin{definition}
	\label{def::stein_variety}
	A complex variety $X$ is called \cdef{Stein} if the following hold:
	\begin{enumerate}[]
		\item \label{def::separation_property_Stein}
		for any $x \neq y \in X$ there exists $f \in H^0(M, \holo_X)$ such that $f(x) \neq f(y)$
		\item \label{def::holomorphically_convex_property_Stein}
		the holomorphically convex hull of any compact set is compact
	\end{enumerate}
\end{definition}
	
	The most natural definition of an algebraic cone manifold involves just the concept of a contraction:
	
	\begin{definition}
		\label{def::alg_cone_contraction}
		Let $X$ be a complex (Stein)\footnote{In fact, it is easy to see that the Stein condition follows a posteriori} variety with a unique, isolated singularity $x \in X$. Suppose $X$ admits a holomorphic contraction with origin in $x$. Then $X$ is called an \cdef{closed algebraic cone manifold}. $X \setminus \{ x \}$ is called an \cdef{open algebraic cone manifold}\footnote{In \cref{def::alg_cone_contraction} we insist on the word ``manifold" because there is also the more general notion of a (closed/open) algebraic cone varity, in which we allow for more than a unique, isolated singularity.}.
	\end{definition}
	
	For a given open cone manifold, there are in fact several closed cone manifolds from which it can be obtained. We will be interested in the following, privileged kinds:
	\begin{definition}
		A complex variety $X$ is called \cdef{geometrically normal} if any locally bounded meromorphic function on $X$ is holomorphic.
	\end{definition}
	
	It turns out that for a given open cone manifold, there exists a unique geometrically normal closed cone manifold which it originates from. To understand this, we need to work in the more general framework of Stein completions.
	
	\begin{definition}
		Let $X$ be a complex variety. A complex Stein variety $\widetilde X$
		 is called \cdef{a weak Stein completion} of $X$ if 
		 there exists a compact set $K \subset \widetilde X$ such that 
		 $X$ is biholomorphic to $\widetilde X \setminus K$.
		 If $\widetilde X$ can be chosen normal, then $\widetilde X$ is called \cdef{the Stein completion} of $X$.
	\end{definition}
	
	\begin{remark}
		\label{rmk::stein_completion}
		If $\widetilde X$ is a weak Stein completion of $X$ such that $\widetilde X$ happens to be geometrically normal, then by \cite[Theorem~6.6]{rossi:vector_fields_analytic_spaces} any holomorphic function on $X$ can be extended to one on $\widetilde X$, while \cite[Satz~1,~p.~378]{forster:theorie_steinischer} guarantees (for any Stein variety, not just in the geometrically normal setting) that $H^0(\widetilde X, \holo_{\widetilde X})$ uniquely determines $\widetilde X$. Therefore, $\widetilde X$ is unique, justifying the definite article in the expression "the Stein completion".
	\end{remark}
	
	Assuming now again that $X$ is an open cone manifold, \cref{rmk::stein_completion} takes care of the uniqueness of a geometrically normal closed cone. The existence is one of the results proven in \cite{ornea_verbitsky::algebraic_cones}, where a characterisation of algebraic cones involving the total space of a line bundle on a projective orbifold is proven and used to this end; by virtue of it, the existence in question becomes a consequence of the Remmert-Stein theorem, \cite[Proposition~4.3]{ornea_verbitsky::algebraic_cones}, and the equivalence between projective normality and normality expressed in \cite[Proposition~8.4]{ornea_verbitsky::algebraic_cones}.
	
	We do not insist on the aforementioned characterisation of algebraic cones, merely recalling it in the proof of \cref{thm::algebraicConesAreZCoversLCK}, which is another perspective on algebraic cones involving LCK manifolds. We use the latter in \cref{sec::products_algebraic_cones}. 
	
	From existing results in the literature, we can establish immediately that algebraic cone manifolds
	correspond to $\Z$-covers of an important class of LCK manifold called LCK manifolds with potential, which have a rich history. The class was introduced in  \cite{ornea_verbitsky::2010LCKwithPotential} with the objective of proving deformation stability, while \cite{ornea_verbitsky::survey} broadened the
	definition to ``automorphic potential on a cover'' and conjectured
	properness is automatic. This was refuted in \cite{ornea_verbitsky::LCKRank}, whence the
	proper/improper dichotomy was introduced via the LCK rank. For an extended history of the topic we refer to \cite{lck_book}.
	For convenience, we recall only the relevant definitions.
	
	\begin{definition}[{\cite[Definition~2.1]{ornea_verbitsky::survey}}]
		\label{def::lckPotential}
		An \cdef{LCK manifold with
		potential} is a complex manifold admitting a K\"ahler covering 
		$(\widetilde M,\widetilde \omega)$ and a
		positive smooth function $\varphi\colon\widetilde{M} \to \R^{>0}$ (the \cdef{LCK potential}) such that
		the deck group of $\widetilde{M}$ acts on $\varphi$ by multiplication by a constant, and
		$\varphi$ is a K\"ahler potential i.e. $d d^c \varphi = \widetilde \omega$.
	\end{definition}
	
	\begin{definition}
		\label{def::lckProperPotential}
		An \cdef{LCK manifold with proper potential} is an LCK manifold with potential admitting a
		K\"ahler cover $(\widetilde M,\widetilde \omega)$ with a potential $\varphi\colon\ \widetilde M \to\R^{>0}$ which is proper. Equivalently (\cite[Remark~2.6]{ornea_verbitsky::LCKRank}) $\widetilde M$ is a $\Z$-cover.
	\end{definition}
	
	The following theorem is an immediate consequence of two relationships established in \cite{ornea_verbitsky::algebraic_cones}, namely that between algebraic cones and total spaces of line bundles over projective orbifolds, 
	and that between these total spaces and LCK manifolds with proper potential.
	\begin{theorem}[{\cite{ornea_verbitsky::algebraic_cones}}]
		\label{thm::algebraicConesAreZCoversLCK}
		Any closed algebraic cone can be obtained as the Stein completion of the K\"ahler $\Z$-cover of an LCK manifold with proper potential. 
		Conversely, the Stein completion of the K\"ahler $\Z$-cover of an LCK manifold with proper potential is always a closed algebraic cone.
		\begin{proof}
			Let $L$ be an ample line bundle over a projective orbifold such that $\Tot^\circ(L)$, the space of non-zero vectors in $L$, is smooth. As in the proof of \cite[Proposition~3.6]{ornea_verbitsky::algebraic_cones}, the existence of a plurisubharmonic function on $\Tot^\circ(L)$ guarantess that if we consider an expanding homothety $\gamma \in \Aut(\Tot^\circ(L))$, the quotient $\Tot^\circ(L)/ \langle \gamma \rangle$ is an LCK manifold with potential.
			
			On the other hand, \cite[Theorem~7.9]{ornea_verbitsky::algebraic_cones} shows that the K\"ahler $\Z$-cover of an LCK manifold with proper potential is a smooth space $\Tot^\circ(L)$ for $L$ an ample line bundle over a projective orbifold as before.
			
			Finally, \cite[Theorem~2.12]{ornea_verbitsky::algebraic_cones} shows that any closed algebraic cone (in the sense of \cref{def::alg_cone_contraction}) is biholomorphic to the Stein completion of a smooth space $\Tot^\circ(L)$ as above.
		\end{proof}
	\end{theorem}

\section{Products of algebraic cones}
\label{sec::products_algebraic_cones}

Let $X_1$ and $X_2$ be two closed algebraic cones and $C_1$, $C_2$ their corresponding open cones. Denote $X:= X_1 \times X_2$ and $C := C_1 \times C_2$.
Our goal is to define an action on $C$ by the group $G = (\C, +)$, such that the quotient $M:= C / G$ is a complex manifold.

Suppose $\gamma \in G \subset \Aut(X)$ satisfies $\gamma\rvert_{X_1} \in \Aut(X_1)$ and $\gamma\rvert_{X_2} \in \Aut(X_2)$ and $\gamma_1 := \gamma\rvert_{X_1}$, $
\gamma_2 := \gamma\rvert_{X_2}$ are both contractions of $X_1$ and, respectively, $X_2$. Then $\{ \gamma_1, \gamma_2 \}$ generates a sub-action of $\mathbb Z^2 \subset G$. Suppose $G/\Z^2$ acts freely on $C$. 

Thus, the quotient $C \rightarrow C/G$ factors through:
$$
	C \rightarrow C/\Z^2 \rightarrow (C/\Z^2)/\T^2 = C/G
$$

By \cref{thm::algebraicConesAreZCoversLCK}, $X_1$ can be seen as the Stein completion of the $\mathbb Z$-cover of some compact LCK manifold with potential, and the same is true for $X_2$. 
Therefore, the quotient $C /\langle \gamma_1 , \gamma_2 \rangle$
is a product of two compact LCK manifolds, say $M_1 \times M_2$. Thus, $C/G$ is a smooth complex manifold, being a quotient of the complex manifold $M_1 \times M_2$ by the free action of $\T^2$.

	
	Given an endomorphism $\varphi$ of a vector space, a vector $v$ is called $\varphi$-finite if the space generated by $\{v, \varphi(v), \varphi^2(v), \ldots \}$ is finite dimensional. Consider as before $X$ a product of closed algebraic cone manifolds and $\gamma \in \Aut(X)$. The pullback $\gamma^*$ is an enodomorphism of $H^0(X, \mathcal{O}_X)$. In this subsection we show how, given a choice of $\gamma$, the $\gamma^*$-finite holomorphic functions determine an algebraic structure. To be precise:
	
	\begin{definition}
		\label{def::algebraicStructure}
		Let $X$ be an analytic variety. An \cdef{algebraic structure on $X$} is a subsheaf $\mathcal Z$ of $\mathcal{O}_X$ having the property that there exists a quasi-projective variety $Z$ and a biholomorphism $f: X \rightarrow Z$ satisfying $f_* \mathcal{Z} \simeq \mathcal{O}^{\textrm{alg}}_Z$, where $\mathcal{O}^{\textrm{alg}}_Z$ denotes the sheaf of algebraic functions. 
	\end{definition}
	
	\begin{theorem}
		\label{thm::algebraicStructureFromContraction}
		Let $X$ be the product of two closed algebraic cone manifolds and let $\gamma$ be an automorphism of $X$ which is a contraction on both factors.
		Then $X$ admits an algebraic structure such that the regular functions are precisely the $\gamma^*$ finite ones. 
	\end{theorem}
	
	To prove \cref{thm::algebraicStructureFromContraction}, we exploit the behaviour of LCK manifolds with potential with respect to Hopf manifolds.
	
	\begin{definition}
		\label{def::HopfManifolds}
		Let $\varphi \in \End(\C^n)$ be a linear automorphism. If $\varphi$ is a contraction i.e. all eigenvalues of $\varphi$ have norm $< 1$, 
		then the quotient $H := (\C^n \setminus \{ 0 \})/ \langle \varphi \rangle$
		is called a \cdef{linear Hopf manifold}.
	\end{definition}
	
	\begin{theorem}[{\cite[Theorem~3.4]{ornea_verbitsky::2010LCKwithPotential}}]
		\label{thm::holoEmbeddingLCKHopf}
		A compact LCK manifold with potential admits a holomorphic embedding to a linear Hopf manifold.
	\end{theorem}
	
	\begin{proof}[{Proof of \cref{thm::algebraicStructureFromContraction}}]
		We adapt the proof of \cite[Theorem~6.3]{ornea_verbitsky::algebraic_cones} to our situation.
		Let $X_i$ be a closed algebraic cone and $M_i$ be an LCK manifold with potential for which $C_i$, the open algebraic cone of $X_i$, is a $\Z$-cover (\cref{thm::algebraicConesAreZCoversLCK}). Let $M_i \rightarrow H_i$ the holomorphic embedding guaranteed by \cref{thm::holoEmbeddingLCKHopf}. Since the deck group of $C_i \rightarrow M_i$ is $\Z$, the embedding lifts to $f: C_i \rightarrow \C^n \setminus \{ 0\}$. The Remmert-Stein theorem guarantees that the closure $C_i^{\textrm{an}}$ of $f(C_i)$ in $\C^n$ is an analytic variety, thus a weak Stein completion of $C_i$; however, it doesn't necessarily coincide with the Stein completion $X_i$ (see \cite[Sections~4~and~8]{ornea_verbitsky::algebraic_cones}). 
		But by \cite[Theorem~6.2]{ornea_verbitsky::algebraic_cones}, $C_i^{\textrm{an}}$ is an algebraic subvarity. We can then consider the algebraic normalization $\widetilde{C_i^{\textrm{an}}}$ of $C_i^{\textrm{an}}$. This can be explicitly obtained by taking the integral closure of the ring of algebraic functions of $C_i^{\textrm{an}}$ in its field of fractions; this integral closure is guaranteed to be a finitely generated $\C$-algebra by Emmy Noether's finiteness theorem, \cite[Corollary~13.13]{eisenbud::commutativeWithView}.
		By \cite[Theorem~II.7.3]{demailly::complexAnalyticDiffGeomBook}, $\widetilde{C_i^{\textrm{an}}}$ is, as an analytic variety, the geometric normalization of $C_i^{\textrm{an}}$. But then, by \cref{rmk::stein_completion}, $\widetilde{C_i^{\textrm{an}}}$ is biholomorphic to $X_i$. 
		
		We obtain in this way a holomorphic embedding of $X_i$ into $\C^{n_i}$ such that the image is an algebraic variety.
		Thus the analytic variety $X = X_1 \times X_2$ embeds into some $\C^N$, where it obtains the structure of an algebraic variety.
		Clearly the algebraic functions on $X$ are precisly the holomorphic functions on $X$ which are polynomial. By \cite[Lemma~6.1]{ornea_verbitsky::algebraic_cones}, since (the lift of $\gamma$ to $\C^N$ is a contraction), $\gamma^*$-finite functions on $\C^N$ are the same as polynomial functions. 
		
		Thus, it suffices to prove that a $\gamma^*$-finite function can be extended to $\C^N$. For this, we consider $I_X$ the sheaf of analytic functions vanishing on $X_i$ seen as an analytic subvariety in $\C^N$. We then have a short exact sequence of analytic sheaves:
		\begin{equation}
			\label{eqn::shortExactSeqIdealSheafSubvarC^N}
			0 \rightarrow I_{X}
			\rightarrow \mathcal{O}_{\C^N}
			\rightarrow \mathcal{O}_{X}
			\rightarrow 0
		\end{equation}
		
		Since $\C^N$ is Stein and $I_{X}$ is coherent by Oka's coherence theorem (\cite[Chapter~II,~5.2]{grauert_remmert::coherentAnalyticSheaves}), Cartan's Theorem B (\cite[Chapter~VII,~A.Theorem~14]{gunning_rossi::analyticFunctions}) entails that 
		${H^1(\C^N, I_X) = 0}$
		. Then the long exact sequence in cohomology given by \cref{eqn::shortExactSeqIdealSheafSubvarC^N} gives a surjective map
		$$
		H^0(\C^N, \mathcal{O}_{\C^N}) \rightarrow H^0(\C^N, \mathcal{O}_X)
		$$
		which completes the proof.
	\end{proof}

\section{Holomorphic tensors on the product of algebraic cones}
\label{sec::holom_tensors_prod_alg_cones}

Let $X_1$ and $X_2$ be two closed algebraic cone manifolds, and $C_1$, $C_2$ be their corresponding open cones. Assume throughout this section that $\dim C_i > 2$.
As before denote $X:= X_1 \times X_2$ and $C := C_1 \times C_2$.
Suppose that $G = \C$ acts biholomorphically on $C$ such that $G$ contains a $\gamma$ which can be extended to a contraction on each $X_i$ at the point $X_i \setminus C_i$.
Consider $X$ as equipped with an algebraic structure determined by $\gamma$-finite functions (\cref{thm::algebraicStructureFromContraction}).

\begin{theorem}
	\label{thm::holomorphicTensorsProdConesFromBase}
	Let $X$ be the product of two closed algebraic cone manifolds equipped with a $G=(\C, +)$-action containing an element $\gamma$ which is a contraction on both factors. Let $C$ be the product of the open cones, $\pi: C \rightarrow C/G$ be the quotient map, and $\mathcal T$ any holomorphic tensor bundle of on $C/G$. 
	
	Suppose each cone participating in the product has dimension $\geq 2$.
	
	Then any holomorphic section of $\pi^* \mathcal T$ is invariant under the Zariski closure of $\langle \gamma \rangle$.
	
	\begin{proof}
		Denote by $S:= X \setminus C$ the singular locus of $X$. 
		
		\newProofStep
		\label{step::extensionOfVectorFieldGivingAction} 
		Let $V$ be the vector field whose flow generates the $G$-action. We show that $V$ extends to a vector field on $X$.
		To begin with, the tangent sheaf $T_X$ is reflexive, since it is the dual of the sheaf of K\"ahler differentials, which is coherent by \cite[Chapter~2.10,~Lemma~p.87~and~description~on~p.88]{fischer::ComplexAnalyticGeometry}. $X$ is normal, being a product of normal varieties.
		Since also $\codim S \geq 2$, we can apply \cite[Sec.~3,~Remarque~2~after~Proposition~4]{serre::prolongement} to $T_X$ to obtain that $\Gamma(X, T_X) \simeq \Gamma(X \setminus S, T_X)$. This shows that $V$ extends to a section of $T_X$ on the whole product $X$.
		
		\newProofStep Let $\mathcal T$ be a holomorphic tensor bundle on $C/G$.
		 We show that $\pi^* \mathcal T$ extends to a coherent sheaf on $X$ and, in fact, that
		 if $\iota: X \setminus S \rightarrow X$ denotes the inclusion, then
		 $\iota_* \pi^* \mathcal T$ is coherent.
		 Note first that we have the following exact sequence:
		 $$
		 0 \rightarrow 
		 		\mathcal{O}_{X \setminus S} \cdot V 
		 	\rightarrow 
		 		T_{X \setminus S} 
		 	\rightarrow 
		 		\pi^* T_M 
		 	\rightarrow 0
		 $$
		Consider $V$ extended to the whole $X$ as in \cref{step::extensionOfVectorFieldGivingAction}. 
		Then by the exact sequence above, 
		$\mathcal{F} := \frac{T_X}{{\mathcal{O}_X} \cdot V}$
		is a sheaf on $X$ which extends $\pi^* T_M$. But $\mathcal{F}$ is a quotient sheaf of a coherent sheaf by a coherent sheaf (\cite[Chapter~2,~Section~5.3,~Proposition~on~pg.~60]{grauert_remmert::coherentAnalyticSheaves}), which is therefore coherent
		(\cite[Appendix.4.2,~Consequence~2]{grauert_remmert::coherentAnalyticSheaves}).
		
		Because tensor products of coherent sheaves are coherent
		(\cite[Appendix.4.4]{grauert_remmert::coherentAnalyticSheaves}), 
		$\pi^* T_M^{\otimes k}$ admits a coherent extension, which must be precisely $\iota_* \pi^* T_M^{\otimes k}$ by \cite[Théorème~1]{serre::prolongement} since $X$ is normal.
		
		It remains to see that $\iota_* \pi^* (T_M^*)^{\otimes l}$ is coherent, which amounts to showing that $\pi^* T_M^*$ has a coherent extension. But the sheaf on $X$ defined as the kernel of the contraction with $V$,
		$$
		\ker \left( i_V : \Omega_X^1 \rightarrow \mathcal{O}_X \right)
		$$
		extends $\pi^* T_M^*$ and is coherent because it is the kernel of a morphism between two coherent sheaves (\cite[Appendix.4.2,~Consequence~2]{grauert_remmert::coherentAnalyticSheaves})
		This together with the theorem of Serre shows that also $\iota_* \pi^* T_M^*$ is coherent. 
		
		\textbf{Step 3.} Let $s$ be a holomorphic section of $\pi^* \mathcal{T}$. 
		Then $s$ is $G$-invariant, thus in particular $\langle \gamma \rangle$-invariant, so by \cref{lemma::sectionsCoherentSheavesProductConesAreInvariant} applied to the coherent sheaf $\iota_* \pi^* \mathcal{T}$, $s$ is invariant under the Zariski closure of $\langle \gamma \rangle$.
	 
	\end{proof}
\end{theorem}

\begin{lemma}
	\label{lemma::sectionsCoherentSheavesProductConesAreInvariant}
	Let $\mathcal{F}$ be an analytic coherent sheaf on a product $X$ of two closed algebraic cone manifolds equipped with a contraction $\gamma$ to the product of the apexes.
	Then any $\gamma$-invariant section of $\mathcal{F}$ is invariant under the Zariski closure of $\langle \gamma \rangle$.
	
	\begin{proof}
		
		Recall that for $x \in X$, the space of $k$-jets of $\mathcal{F}$ is by definition $\mathcal{F}_x^k := \frac{\mathcal{F}_x}{\mathfrak{m}_x^k \mathcal{F}_{x}}$, where $\mathcal{F}_x$ denotes the space of germs of sections of $\mathcal{F}$ at $x$, and $\mathfrak{m}_x$ denotes the maximal ideal of holomorphic functions vanishing at $x$. There are natural maps $\mathcal{F}_x^{k+1} \rightarrow \mathcal{F}_x^k$, giving an inverse system, of which we denote the inverse limit by $\widehat{\mathcal{F}_x}$ (\cite[Chapter~10]{atiyah_macdonald::introduction_commutative}).
		
		Let now $x \in X$ be the point $(0_{X_1}, 0_{X_2})$ where $\{ 0_{X_i} \}$ is the singular set of $X_i$. Denote by $H^0_\gamma (X, \mathcal{F})$ the $\gamma$-invariant sections. We have natural maps:
		
		\begin{center}
			\begin{tikzcd}
				H^0_{\gamma}(X, \mathcal F)
				\ar[r]
				\ar[rr,out=-30,in=210,swap,"\Phi"] & 
				\mathcal F_x 
				\ar[r] & 
				\widehat {\mathcal F_x} 
			\end{tikzcd}
		\end{center}
		
		Observe that the first map from left to right is injective. Indeed, let $s \in H^0_{\gamma}(X, \mathcal F)$ having vanishing germ at $x$ and suppose by absurd that $\supp(s)$ is nonempty; say $y \in \supp (s)$. Since $\supp(s)$ is closed, we can apply the definition of $\gamma$ being a contraction, \cref{def::contraction}, to the compact set $\{ y \}$ and the open set $X \setminus \supp (s)$, which contains $s$ because $s_x = 0$, to obtain that $\gamma^n(y) \in X \setminus \supp(s)$ for large enough $n$. But since $s$ is $\gamma$-invariant, $\supp(s)$ is $\gamma$-invariant, so $y^n \in \supp(s)$ for any $n$, a contradiction. 
		
		Now we explain why the second map from left to right is injective. It is known that $\mathcal{O}_{X, x}$ is Noetherian (\cite[Ch.~2,~Sec.~B,~Theorem~9]{gunning_rossi::analyticFunctions}).
		Furthermore, the $\mathcal{O}_{X,x}$-module $\mathcal{F}_x$ is finitely generated because the sheaf $\mathcal F$ is coherent. Therefore, we can apply
		the corollary of Krull's theorem \cite[Corollary~10.19]{atiyah_macdonald::introduction_commutative} for the maximal ideal $\mathfrak{m}_x$ in the Noetherian local ring $\mathcal{O}_{X,x}$, guaranteeing that $\bigcap_k \mathfrak{m}_x^k \mathcal{F}_x = \{ 0 \}$ i.e. the second map from left to right is injective as well.
		
		Let now $s \in H^0_\gamma(X, \mathcal{F})$. Then for any $k$, the $k$-jet of $s$, $s^k_x$, is a $\gamma$-invariant vector in the finite-dimensional space $\mathcal{F}^k_x$; but $\stab(s^k_x)$ is the preimage of the closed point $\{s^k_x\}$ under the orbit morphism $g\mapsto g s^k_x$, hence Zariski-closed. Since it contains $\gamma$, it contains
		$\zarClosure{\gamma}$. Hence $\Phi(s)$ is invariant under the Zariski closure, and thus, by injectivity of $\Phi$, so is $s$ itself.
	\end{proof}
\end{lemma}

\section{Sasaki Manifolds}
\label{sec::sasakianManifolds}

The goal of \cref{sec::sasakianManifolds} and \cref{sec::sectionRingAmpleness} is to gather supportive material and results for \cref{thm::holomorphicTensorsProdSasaki}.

\begin{definition}
	\label{def::Sasaki}
	Let $(S,g)$ be an odd-dimensional Riemannian manifold and $(C(S):=(S \times \R^{>0}, g_{C(S)}= dr \otimes dr + r^2 g)$, $r\in\R^{>0}$, its Riemannian cone. 
	We say $S$ is a \cdef{Sasaki manifold} if there exists a K\"ahler structure $(J, \omega, g_{C(S)})$ on $C(S)$ such that the homothety map $h_\lambda: C(S) \rightarrow C(S)$, $h_\lambda(p,r):= (p, \lambda r)$ is holomorphic and satisfies $h_\lambda^* \omega = \lambda^2 \omega$ for each $\lambda \in \R^{>0}$.
\end{definition}

We denote by $E := r \frac{d}{dr}$ the \cdef{Euler field} on $C(S)$ and by $\xi:= JE$ the \cdef{Reeb field}. We also denote by $\xi$ the vector field $\xi\rvert_{r=1}$ on $S$ seen as $S \times \{1 \} \subset C(S)$.
 
 For a Sasaki manifold $S$, define an action of $(\C, +)$ on $C(S)$ by putting for $z \in \C$:
\begin{equation}
	\label{eqn::defActionCone}
 z \cdot p := \phi_1^{\Re(z) E + \Im(z) \xi} (p)
\end{equation}
where (in \eqref{eqn::defActionCone} and henceforth) $\phi_t^X$ denotes the flow of the vector field $X$ at time $t$.

\begin{definition}
	\label{def::quasiregReeb}
	Let $S$ be a Sasaki manifold. We say that a Reeb field $\xi$ is \cdef{quasiregular} if the orbits of $\xi$ are compact.
\end{definition}
	
\begin{remark}
	\label{rmk::wadsley}
	In \cite{wadsley::geodesicFoliationsCircles}, Wadsley gives necessary and sufficient condition for a foliation with compact leaves to come from a circle action (\cite[Theorem~2.6.12]{boyer_galicki::SasakianGeometryBook}).
	In fact, Wadsley's proof shows more, namely that the fundamental field of the action can be taken to be a normalized field with respect to a contact form (in particular, the Reeb field on a Sasaki manifold) and for such a field $\xi$ there exists a common time $T_0$ such that $\phi^\xi_{T_0} = \identity$. This follows from the fact that the least period map is constant on any connected component of the complement of its discontinuity set (\cite[Lemma~4.3]{wadsley::geodesicFoliationsCircles}), the fact that for such a constant $c$ in the image of the period map we have $\phi^\xi_c = \identity$ on the set where the period map is locally unbounded (\cite[Proposition~4.5]{wadsley::geodesicFoliationsCircles}), as well as the fact that this latter set is empty (\cite[Lemma~4.8]{wadsley::geodesicFoliationsCircles}).
\end{remark}

\begin{remark}
	\label{rmk::C^*ActionQuasiregCone}
	By rescaling the quasiregular $\xi$, we will always assume $T_0$ from \cref{rmk::wadsley} is equal to $2\pi$. Then from the action \eqref{eqn::defActionCone} we obtain a $\C^*$ action on the open cone $C$ by:
	\[
		a(\exp(t), p) := t \cdot p
	\]
	which is independent of the choice of the logarithm because $T_0 = 2\pi$.\footnote{
		It is interesting to contrast this with the more general result stating the existence of a logarithm of a power of the contraction generating the deck group of an LCK manifold with proper potential (\cref{def::lckProperPotential}), namely \cite[Theorem~14.6]{lck_book}. Remarkably, the restriction to powers is unavoidable even in the Vaisman case (\cite[Example~14.8]{lck_book}).
	}
\end{remark}

\begin{remark}\label{rmk::constructionOfConeOrbifoldChart}
	Since a Reeb vector field vanishes nowhere, the isotropy groups of the $S^1$ action discussed in \cref{rmk::wadsley} are finite cyclic. For a quasiregular (\cref{def::quasiregReeb}) Reeb field $\xi$, this translates to the fact that the local uniformising groups, denoted $\Gamma_p$ for $p \in S$, of the quotient orbifold $Y:= S/\langle \xi \rangle$ (\cite[Theorem~7.1.3]{boyer_galicki::SasakianGeometryBook}) are finite cyclic.
	
	There is a natural way to construct orbifold charts for $Y$ starting from the open cone $C$, which we will need and we'll recall now briefly, adapting \cite[Chapter~II,~Theorem~5.4]{bredon::introductionCompactTransformationGroups} for the complex holomorphic case.
	Let $p \in C$. Take a $\Gamma_p$ invariant metric on $T_pC$. With respect to it, we take $N \subset T_pC$ the orthogonal complement of the tangent space at $p$ to the orbit of $p$. Taking a $T_p C$-valued holomorphic chart around $p$ and averaging it with the right actions of $\Gamma_p$, we obtain after shrinking the chart domain a $\Gamma_p$-equivariant chart $\phi$ and we put $V_p := \phi^{-1}(N \cap B_1(0))$ where $B_1(0)$ is a ball in $T_pC$.

	The key point is that up to shrinking $V_p$, $\C^* \times_{\Gamma_p} V_p:= (\C^* \times V_p)/\Gamma_p$ are local tubular neighborhoods of the orbit of $p$. Indeed, $\tau: \C^* \times_{\Gamma_p} V_p \to C$, $\tau([z,x]) := z\cdot x$, is a well defined $\C^*$-equivariant holomorphic map surjecting on $\C^* V_p$. Using that $d \tau_{[1,p]}$ is bijective and the properness of the $\mathbb C^*$-action, it can be seen that, up to shrinking $V_p$, $\tau$ is also injective and $\Gamma_{p'}\subseteq\Gamma_p$ for every $p'\in V_p$. 
	
	Denoting $\pi: C \rightarrow Y$, $\pi|_{V_p}$ induces a homeomorphism
	$V_p/\Gamma_p\xrightarrow{\sim}U_p:=\pi(\C^* V_p)$;
	and it can be seen that the triples $\{(V_p,\Gamma_p,\pi\rvert_{V_p})\}_{p \in \C}$ constitute an orbifold atlas (\cite[Definition~4.1.1]{boyer_galicki::SasakianGeometryBook}); these are the preferred charts we will work with.
\end{remark}

\begin{remark}
	\label{rmk::lineOrbibundlesDef}
	On the local uniformising charts $(V_p, \Gamma_p)$ from \cref{rmk::constructionOfConeOrbifoldChart}. Define 
	$B^L_{V_p}:=V_p\times\C\ \xrightarrow{\,\mathrm{pr}_1\,}\ V_p$ 
	and $h_{V_p}^L: \Gamma_p \hookrightarrow GL(\C) \simeq \C^*$ the natural embedding, with the $\Gamma_p$ action on $V_p \times \C$ given by $\gamma (x, v) := (\gamma^{-1}x, \chi(\gamma) v)$. By a routine check, these glue to a line orbibundle (\cite[Definition~4.2.7]{boyer_galicki::SasakianGeometryBook}) $L$.
	
	Similarly, $B_{V_p}^C := V_p \times \C^* \xrightarrow{\,\mathrm{pr}_1\,}\ V_p$ with $h^C_{V_p} = h^L_{V_p}$ with the same action of $\Gamma_p$ as on $B^L_{V_p}$ give a principal $\C^*$-orbibundle.
	Due to the identification $\tau$ from \cref{rmk::constructionOfConeOrbifoldChart}, the total space of this principal $\C^*$-orbibundle over $U_p$ is
	$
	\Tot \rvert_{U_p} =  V_p \times_{\Gamma_p} \C^*
	\xrightarrow{\stackrel{\tau}{\sim}}
	 \C^* V_p \subset C
	 $
	 Therefore, the total space of this bundle is the open cone $C$, thus we also call it $C$. By construction $\Fr(L) = C$ and so $\Tot^\circ(L) = C$, where $\Tot^\circ$ denotes the total space without the $0$ section.
\end{remark}

\section{The ring of sections and ampleness}
\label{sec::sectionRingAmpleness}

With the notations from \cref{sec::sasakianManifolds}, we study now the central objects of \cref{sec::holomTensorsProdSasaki}, namely the rings:

\begin{equation}
	\label{eqn::gradedRingsSections}
	R^m := H^0(Y, (L^*)^{\otimes m}), \quad R := \sum_{m = 0}^\infty R^m
\end{equation}

Importantly, in \eqref{eqn::gradedRingsSections} sections are taken in the orbibundle sense (\cite[Definition~4.2.9]{boyer_galicki::SasakianGeometryBook}).
Consider also the sheaves on the complex space $Y$ defined as:
\begin{equation}
	\label{eqn::sheavesInvariantFunctions}
	\mathcal{F}_m(U)
	:=
	\{f\in\mathcal{O}(C|_{U}):f(\lambda v)=\lambda^mf(v)\ \text{on}\ C|_U\}
\end{equation}
and the corresponding graded ring of homogeneous functions on the open cone $C$:

\begin{equation}
	\label{eqn::gradedRingsFunctions}
	 A_m := 
	 H^0(Y, \mathcal{F}_m)
	 , \quad A:= \sum_{m=0}^\infty A_m
\end{equation}

Work on the level of the orbifold charts $(V_p, \Gamma_p, U_p)$ from \cref{rmk::constructionOfConeOrbifoldChart}. 
On $V_p$ a section of $(L^*)^{\otimes m}$ is by construction of $L$ (\cref{rmk::lineOrbibundlesDef}) a family of holomorphic functions
$\widehat s_p \in \holo(V_p)$ 
satisfying:
\begin{equation}
	\label{eqn::equivarianceSectionsOrbibundle}
	\widehat s_p(\gamma\cdot x) = h^L_{V_p}(\gamma)^{-m} \widehat s_p(x)
	\quad
	\forall \gamma \in \Gamma_p,\ x \in V_p
\end{equation}
and compatible under the orbifold injections.

The following dual perspective on $R$ will be indispensable.
\begin{lemma}
	\label{lemma::sectionRingsAreFunctions}
	
	The mappings:
	\begin{align*}
		ev : 
		\{\widehat s\in\mathcal{O}(V_p):\widehat s\ \text{satisfies \eqref{eqn::equivarianceSectionsOrbibundle}}\}
		&\rightarrow 
		A_m(V_p \times_{\Gamma_p} \C^* ) \\
		ev(\widehat s)([x, v]) &:= \widehat s (x) v^m
	\end{align*}
	are well-defined $\C$-linear
	bijections gluing to bijections $ev: H^0(Y,(L^{*})^{\otimes m})\to A_m$ for all $m \geq 0$. They assemble to a ring isomorphism $R \simeq A$.
	\begin{proof}
		The only nonstandard check is surjectivity.
		Denote $q: V_p \times \C^* \rightarrow (V_p \times_{\Gamma_p} \C^*)$ the quotient map, let $f \in A_m(V_p \times_{\Gamma_p} \C^* )$ and put $\tilde f := q^* f$. 
		
		Consider the holomorphic function $h:=\widetilde f(x,\cdot)$ on $\C^*$.
		It has a Laurent expansion $h(v)=\sum_{k\in\mathbb{Z}}a_k v^k$, convergent for all
		$v\in\C^*$.
		
		Since $h(\lambda v)=\lambda^m h(v)$, by expansion of both sides in Laurent series and uniqueness of coefficients we obtain: 
		\[
		a_k \lambda^{k} = \lambda^{m}a_k
		 \qquad \forall k \in \mathbb{Z}
		, \ \lambda\in\C^*
		\]
		which entails $a_k=0$ for $k\neq m$.  Hence $h(v)=a_m v^m$ and
		$a_m = \widetilde f(x,1)$.  Define $\widehat s(x):=\widetilde f(x,1)$, which is the restriction of the holomorphic $\widetilde f$ to the
		slice $V_p \times \{1\}$, thus holomorphic. This shows
		$\widetilde f(x,v)=\widehat s(x)v^m$; these agree on overlaps because $\tilde f$ comes from a global $f$.
	\end{proof}
\end{lemma}

By \cite[Theorem~7.1.3,~(ii)~and~(vi)]{boyer_galicki::SasakianGeometryBook} applied for the reduction of the structure group $\C^*$ of $L^*$ to $S^1$, this reduction has a connection whose curvature is the pullback of a K\"ahler form on $Y$ by the projection to $Y$. Therefore $L^*$ is positive.
By the Kodaira embedding theorem for orbifolds (\cite[p.~427,~Global~Imbedding~Theorem]{baily::VManifoldsInProjSpace}), $L^*$ is orbiample and we find an embedding $\iota: Y \rightarrow \C P^d$ with $\iota (Y)$ a normal algebraic variety and $(L^*)^{\otimes N} \simeq \iota^* \mathcal{O}(1)$ for some large $N$ dividing $\nu := \lcm ( \ord (\Gamma_p) )_{p \in S}$. 

\begin{proposition}
	\label{proposition::sectionRingFinitelyGenerated}
	
	$R$ is a finitely generated $\C$-algebra.
	\begin{proof}
		In short, the idea of the proof is that the powers of $L^*$ that are divisors of $\nu$ are actual line bundles on $Y$, while for the rest of the powers we show that, when large enough, they are finitely generated modules over the sections of the powers of $N$.
		
		Consider the sheaves $\mathcal{F}_m$ of \cref{eqn::sheavesInvariantFunctions} 
		Since $Y \subset \C P^d$ is algebraic, by 
		\cite[{n\textsuperscript{o}~12, Th.~1-3,~pp.~19-20}]{serre::GAGA}, 
		there exists a unique algebraic sheaf $\mathcal{F}_m^{\alg}$ having the same cohomology groups as $\mathcal{F}_m$, in particular the same sections.
		
		\newProofStep \label{step::secRingFG/LazarsfeldUsage}
		Consider $
		B:=R^{(N)}=\bigoplus_{k\ge0}R^{kN}
		$.
		Since $(L^*)^{\otimes kN}$ is an actual line bundle on $Y$ for $k \geq 1$ by
		\cite{baily::VManifoldsInProjSpace}, we can identify $\mathcal{F}_{kN}$ with $(L^*)^{\otimes kN}$ via \cref{lemma::sectionRingsAreFunctions}, so $\mathcal{F}_{kN}^{\alg}$ is also an ample line bundle, so
		$
		B = \bigoplus_{k \geq 0} H^0 (Y, \mathcal{F}_{kN})
		= \bigoplus_{k \geq 0} H^0 (Y, \mathcal{F}^{\alg}_{kN})
		$
		is a finitely generated $\C$-algebra by
		\cite[Ex.~2.1.30]{lazarsfeld::positivityInAlgGeomI}.
		
		\newProofStep \label{step::secRingFG/cosetModules}
		For $0 \leq a<N$ put
		$M_a:=\bigoplus_{k \geq 0} R^{a+kN}$.
		
		Since $(L^*)^{\otimes N} \simeq \iota^* \holo(1)$:
		\[
		R^{a+kN} = H^0 (Y, \mathcal{F}^{\alg}_a \otimes \holo_Y(k) )
		\]
		
		By \cite[Theorem~II.5.17]{hartshorne::algGeom} applied to the coherent sheaf $\mathcal{F}^{\alg}_a$, there is an integer $k_0$ such that 
		$\mathcal{F}^{\alg}_a \otimes \holo_Y(k_0)$ can be generated by a finite number of global sections - say $P$ such sections - which give an exact sequence
		\[
		0
		\to
			\mathcal K \to \holo_Y^{\oplus P} 
		\to 
			\mathcal{F}_a^{\alg}
			\otimes \holo_Y(k_0)\to 0
		\]
		 with
		$\mathcal K$ coherent (\cite[Appendix.4.2,~Consequence~2]{grauert_remmert::coherentAnalyticSheaves}).
		Then, twisting by $\holo_Y(k)$ and taking the long exact sequence in cohomology we obtain that the cokernel of
		\[
		H^0(Y,\holo_Y(k))^{\oplus P}\ 
		\rightarrow
		H^0\bigl(Y,\mathcal{F}_a^{\mathrm{alg}}\otimes\holo_Y(k_0+k)\bigr)
		\]
		injects into $H^1(V,\mathcal K \otimes \holo_Y(k))$. But by \cite[III.5.2(b)]{hartshorne::algGeom}, for some $k_1$ and every $k \geq k_1$ we have $H^1(V,\mathcal K \otimes \holo_Y(k)) = 0$. So for $k \geq k_1$ we have that
		$
		\bigoplus_{k \geq k_0 + k_1} 
		H^0\bigl(Y,\mathcal{F}_a^{\mathrm{alg}}\otimes\holo_Y(k))
		$ is a finitely generated $B$-module
		since by definition 
		$H^0(Y,\holo_Y(k))^{\oplus P} = (B^k)^{\oplus P} = (R^{kN})^{\oplus P}$ 
		
		Therefore, $M_a$ is
		generated as a $B$-module by the generators of $\mathcal{F}^{\alg}_a \otimes \holo(k_0)$ and the union of bases in $R^{a+jN}$ for $j \leq k_0 + k_1$.
		
		\newProofStep \label{step::secRingFG/finiteGenerationEachPiece}
		
		Since $R^{a+jN}$ are finitely many pieces for $j \leq k_0 + k_1$, the proof is complete if we know that each piece $R^m$ is finite dimensional over $\C$. This follows from 
		\cite[III.5.2(a)]{hartshorne::algGeom}.
	\end{proof}
\end{proposition}

\section{Holomorphic tensors on products of Sasaki manifolds}
\label{sec::holomTensorsProdSasaki}

Once choices of Reeb fields are made and fixed (\cref{sec::sasakianManifolds}), the product of two Sasaki manifolds bears an entire family of complex structures indexed by a complex nonreal parameter. Each of these complex structures acts nondiagonally on the distribution generated by the two Reeb fields, while on the distributions transverse to each of the Reeb fields it acts like the complex structure on the cone above the corresponding Sasakian. These were first introduced by Watson on Sasaki manifolds (\cite{watson}), generalising the construction of Tsukada (\cite{tsukada::eigenvalues}) and called Calabi-Eckmann-Morimoto after \cite{morimoto} and \cite{calabi_eckmann}. This family was studied extensively in \cite{andr_tolch::prodSasaki}.

In \cite{marchidanu::prodSasaki} we proposed an implicit construction of a family of complex structures, indexed also by a complex nonreal parameter, which is defined as follows. 

\begin{remark}
	\label{rmk::complexStructsProdSasaki}
	Let $\alpha \in \C \setminus \R$ and put $G_\alpha := \{ (t, \alpha t) : t \in \C \} \subset \C \times \C$.
	Let $S_1$ and $S_2$ bet Sasaki manifolds and $\pi_\alpha : C(S_1) \times C(S_2) \rightarrow (C(S_1) \times C(S_2))/G_\alpha$. Then \cite[Theorem~3.1]{marchidanu::prodSasaki} shows that 
	$(C(S_1) \times C(S_2))/G_\alpha \simeq S_1 \times S_2$ and this identification endows $S_1 \times S_2$ with a complex structure making $\pi_\alpha$ a holomorphic submersion.
\end{remark}

\begin{remark}
	\label{rmk::conesOverSasakiAreAlgebraicCones} 
	By \cite[Structure~Theorem]{ornea_verbitsky::structureTheoremVaisman}, the cone $C(S)$ over a Sasaki manifold $S$ is the $\Z$-cover of a Vaisman manifold, which is LCK with potential (\cite{ornea_verbitsky::convexShells}). Hence, by \cref{thm::algebraicConesAreZCoversLCK}, $C(S)$ is an open algebraic cone in the sense of \cref{def::alg_cone_contraction}. 
\end{remark}

Now we can prove:

\begin{theorem}
	\label{thm::holomorphicTensorsProdSasaki}
	Let $S$ be the product of two Sasaki manifolds. Let $\mathcal{T}$ be a tensor bundle on $S$ and $\varphi \in \Gamma(S, \mathcal{T})$ be a holomorphic tensor. Then $\varphi$ is invariant under the flows of the Reeb fields of the two Sasaki manifolds.
	
	\begin{proof}
		Denote by $S_i$ each Sasaki manifold and by $C_i$ each Sasakian's open cone.
		Let $\gamma_i$ be contractions to the apexes of the $C_i$'s.
		Then $\gamma := \gamma_1 \times \gamma_2$ is a contraction on $C:= C_1 \times C_2$
		(more precisely, on the product $X$ of the Stein completions $X_i$ of the $C_i$'s).
		By \cref{thm::holomorphicTensorsProdConesFromBase} - which is applicable to $C$ by \cref{rmk::conesOverSasakiAreAlgebraicCones} - as well as the natural compatibility of the complex structure on $S_1 \times S_2$ and the one on $C$ as described in \cref{rmk::complexStructsProdSasaki},
		to show that $\varphi$ is invariant under one flow of one of the Reeb fields it suffices
		to find a $\gamma$ such that the Zariski closure of $\langle \gamma \rangle$
		contains the desired flow.
		
		To achieve this goal, we need to materialise a perspective on the algebraic structure guaranteed by \cref{thm::algebraicStructureFromContraction}, a perspective which involves the Reeb fields explicitly. We do this in several steps.
		
		\newProofStep\label{step::actionOnCones}
		Denote by $\xi_i$ the Reeb field on $S_i$. The isometry group $\Iso(S_i)$ is compact and the closure of the $1$-parameter group generated by the flows of $\xi_i$ inside it is commutative and connected. Therefore, by \cite[Theorem~11.2]{hall::LieGroups}, this closure is a real torus $T^{k_i} = (S^1)^{k_i}$. 
		
	    By \cite[Theorem~7.1.10]{boyer_galicki::SasakianGeometryBook}, 
		we can find quasi-regular Reeb fields arbitrarily close to $\xi_i$ in its Sasakian cone,
		which consists of the Reeb vector fields which are compatible with the same underlying Sasakian CR structure
		(\cite[Section~8.2.3]{boyer_galicki::SasakianGeometryBook}).
		Let $\xi_i^{\quasireg}$ be a quasi-regular Reeb field inside the Lie algebra of $T^{k_i}$.
		
		The complexified torus $T^{k_i}_{\C} = (\C^*)^{k_i}$ acts on $C_i$ in the following way. We see the Lie algebra $T^{k_i}_{\C}$ as $\LieAlg{t}_i \oplus \sqrt{-1} \LieAlg{t}_i$ - where $\LieAlg{t}_i$ is the Lie algebra of $T^{k_i}$ - and we bid $\sqrt{-1}$ act as the complex structure $J_i$ on the cone $C_i$. 
		In other words, if $V, W \in \LieAlg{t}$ then, denoting by $\Phi^V_t$ the flow of the vector field $V$ at time $t$, $V+\sqrt{-1}W$ gives an action by 
		$\exp(V+J_i W)=\Phi_1^{V+J_i W}$ 
		on $C_i$.
		Because $\exp: \LieAlg{t} \oplus \sqrt{-1} \LieAlg{t} \rightarrow T^{k_i}_\C$ is surjective, this defines indeed an action of $T^{k_i}_\C$.
		
		Since the flows of $\xi_i$ are Sasakian automorphisms on the Sasaki manifold, $T^{k_i}$ also consists of Sasakian automorphisms by \cite{tanno::automorphismGroups}. In particular, elements of $T^{k_i}$ lift to biholomorphisms on $C_i$ and thus $T^{k_i}_\C$ acts by biholomorphisms by definition.
		This shows moreover that the action $T^k_\C := T^{k_1}_\C \times T^{k_2}_\C$ itself is holomorphic as a map $T^k_\C \times C \rightarrow C$. Indeed, note that for $V \in \LieAlg{t}_i$ we have that the mapping $s+\sqrt{-1} t \mapsto \phi_s^V \phi_t^{JV}(q)$ is holomorphic in each $q_i$ because $[V, J_iV] = 0$ and $L_V J_i = 0$. Since $\exp: \LieAlg{t}_i \rightarrow T^{k_i}_\C$ is surjective, the map $T^k_\C \times \{ q \} \rightarrow C$ is locally a composition of maps $s+\sqrt{-1} t \mapsto \phi_s^V \phi_t^{JV}(q)$ for each $q$ - once we pick bases of vector fields $V$ in each $\LieAlg{t}_i$ - and it's also obviously holomorphic in the $q$ variable.
		
		\newProofStep\label{step::passageLineBundles}
		For each $S_i$, consider the orbifold $Y_i := S_i/ \langle \xi_i^{\quasireg} \rangle$ and a line orbibundle $L_i$ over it, defined as in \cref{rmk::lineOrbibundlesDef}.
 
		The central objects of the ensuing argument are the graded rings $R_i$ of equation \eqref{eqn::gradedRingsSections} corresponding to each $S_i$.
		
		We will work with the ring $R := R_1 \otimes_\C R_2$.
		A key fact is that $R$ is a finitely generated $\C$ algebra, being a product of two finitely generated ones (\cref{proposition::sectionRingFinitelyGenerated}).
		As will become transparent through the proof, for our purposes we cannot pick an arbitrary set of generators. We must pick the generators in such a way that the embedding they give (see \cref{step::embeddingViaGenerators}) diagonalises the action of $T^k_{\C}$. Fortunately, this can be done, as explained in \cref{step::weightDecomposition}.
		
		\newProofStep
		\label{step::weightDecomposition}
		The goal now is to study the action of $T^{k_i}_\C$ on each of the pieces 
		$R_i^m := H^0(Y, (L_i^*)^{\otimes m})$.
		Since $Y$ is a compact complex space and all tensor powers of $L_i^*$ are coherent analytic sheaves, the Cartan-Serre finiteness theorem (\cite[Chapter~VI]{remmert_grauert::SteinSpaces}) 
		(or \cref{step::secRingFG/finiteGenerationEachPiece}
		of \cref{proposition::sectionRingFinitelyGenerated}
		) guarantees that each $R_i^m$ is a finite dimensional $\C$-vector space.
		
		Let $V:= R_i^m$ and consider the action $\rho: T^{k_i}_\C \rightarrow GL(V)$. Restricting $\rho$ to the compact torus $T^{k_i}$, we can apply the spectral theorem to simultaneously diagonalize all the elements in $\rho(T^{k_i})$. Since $T^{k_i}$ is generated by its Lie algebra and the action of $T^{k_i}_\C$ is also defined through $\exp$ via linear extension on the Lie algebra of $T^{k_i}_\C$, we obtain this way a decomposition of $V$ into simultaneous eigenvectors for $\rho(T^{k_i}_\C)$. More precisely, for $\chi$ a character of $T^{k_i}_\C$ we denote:
		$$
			V_\chi := \{ v \in V: \rho(g) v = \chi(g) v, \quad \forall g \in T^{k_i}_\C \}
		$$
		Then the decomposition is:
		\begin{equation}
			\label{eqn::weightDecomposition}
			V = \bigoplus_{\chi \in \Hom(T^{k_i}_\C, \C^*)} V_\chi
		\end{equation}
		
		Then we simply pick all the generators of $R_i$ so that each of them belongs to some weight space $V_\chi$, where $V$ is some $R_i^m$. The relevance of this choice will become transparent in \cref{step::actionOnEmebedding}.
		
		\newProofStep 
		\label{step::SteinTubes}
		
		We now show how to embed the product $C$ holomorphically into $\C^k \setminus \{ 0 \}$ using $R := R_1 \otimes R_2$.
		
		First, fix one of the $S_i$ and denote by $E$ the zero section of $\Tot(L_i)$. It follows easily from (\cref{def::Sasaki}) that $r^2$ is a K\"ahler potential on the cone, so $r^2$ is smooth and strictly plurisubharmonic on the manifold $C_i$.
		
		Fix $\varepsilon > 0$. First we consider the tube $T_\varepsilon := \{ r < \varepsilon \}$ is relatively compact, and its boundary $\{ r = \varepsilon \}$ lies in $C_i$, where $r^2 - \varepsilon^2$ is a strictly plurisubharmonic defining function; hence $T_\varepsilon$ is strongly pseudoconvex and therefore holomorphically convex (\cite[\S1,~Satz~3]{gauert::uberModifikationen}). Let $\rho : T_\varepsilon \rightarrow \Sigma_\varepsilon$ be its Remmert reduction onto a Stein space, with $\holo_{\Sigma_\varepsilon} \xrightarrow{\sim} \rho_* \holo_{T_\varepsilon}$ (\cite[\S2,~Satz~1]{gauert::uberModifikationen}). Since $r^2$ is strictly plurisubharmonic on $C_i$, the maximum principle for (pluri)subharmonic functions leaves $C_i$ without compact analytic subsets of positive dimension; thus $E$ is the maximal compact analytic subset of $T_\varepsilon$, and $\rho$ contracts it to a single point $o$, with $\rho^{-1}(o) = E$ and $\rho : T_\varepsilon \setminus E \xrightarrow{\sim} \Sigma_\varepsilon \setminus \{ o \}$ a biholomorphism by the properties of the Remmert reduction.
		
		As $\Sigma_\varepsilon$ is Stein, $\holo(T_\varepsilon) = \rho^* \holo(\Sigma_\varepsilon)$ separates the points of $T_\varepsilon \setminus E$.
		
		\newProofStep
		\label{step::weightDecompositionIntegral}
		The circle action on $\Tot(L_i)$, denote it $e^{\sqrt{-1}\theta} \cdot v$, fixes $E$ and preserves $T_\varepsilon$. For $g \in \holo(T_\varepsilon)$ and $m \in \Z$ we set
		\[
		g_m(v) := \frac{1}{2\pi} \int_0^{2\pi} g(e^{\sqrt{-1}\theta} \cdot v) \, e^{-\sqrt{-1}m\theta} \, d\theta .
		\]
		Pulling back to an orbifold chart $(V_p, \Gamma_p, U_p)$ as in \cref{rmk::constructionOfConeOrbifoldChart}, $g$ has a fibrewise Taylor expansion $\sum_{k \geq 0} c_k(x) v^k$ and $g_m$ is its weight-$m$ part $c_m(x) v^m$; hence  $g = \sum_{m \geq 0} g_m$ locally uniformly, with the Cauchy bound $\sup_{\{ r \leq \tau \}} |g_m| \leq (\tau / \tau')^m \sup_{\{ r \leq \tau' \}} |g|$ for $\tau < \tau' < \varepsilon$. Each $g_m$ is a weight-$m$ function on $T_\varepsilon \cap C_i$; it extends uniquely to such a function on all of $C_i$, namely $g_m^\sharp(v) := t^{-m} g_m(t \cdot v)$ for any $0 < t < \varepsilon / r(v)$, and by \cref{lemma::sectionRingsAreFunctions} this $g_m^\sharp$ lies in $H^0(Y_i, (L_i^*)^{\otimes m})$, the graded piece $R_i^m$ of $R_i$. Carried out in each factor for the two circle actions on $C_1 \times C_2$, the same average decomposes a function by bidegree.
		
		\newProofStep
		\label{step::embeddingViaGenerators}
		To arrive at the desired embedding, we use rescalings of what we obtained in \cref{step::SteinTubes}. 
		
		To show that points can be separated, let $v \neq w$ in $C_i$, choose $t > 0$ with $tv, tw \in T_\varepsilon \setminus E$.
		By \cref{step::SteinTubes} some $g \in \holo(T_\varepsilon)$ has $g(tv) \neq g(tw)$, so $g_m(tv) \neq g_m(tw)$ for some $m$, necessarily $m \geq 1$ as $R_i^0 = \C$. Then $f := g_m^\sharp \in R_i^m$ satisfies $f(v) = t^{-m} g_m(tv) \neq t^{-m} g_m(tw) = f(w)$. In particular, applying separation to $v \neq 2 \cdot v$ yields a homogeneous $f \in R_i^m$, $m \geq 1$, with $f(v) \neq 2^m f(v)$, so $f(v) \neq 0$.
		
		To show that the injective map we thus obtain is an immersion, let $0 \neq \zeta \in T_v C$ and pick $t$ with $tv \in T_\varepsilon \setminus E$. By \cref{step::weightDecompositionIntegral}, $g = \sum_m g_m$ can be differentiated termwise, so $dg|_{tv} = \sum_m d(g_m^\sharp)|_{tv}$. Supposing now by absurd that $df|_v(\zeta) = 0$ for every homogeneous $f \in R_i$, then $d(g_m^\sharp)|_{tv}(d (w \mapsto tw) \zeta) = t^m \, d(g_m^\sharp)|_v(\zeta) = 0$ for all $m$, hence $dg|_{tv}(d(w \mapsto tw) \zeta) = 0$ for all $g$, which is impossible since $d(w \mapsto tw) \zeta \neq 0$. So the differentials at $v$ of homogeneous elements of $R_i$ span $T_v^* C_i$.
		
		For a section $s$ of $(L_i^*)^{\otimes m}$ let $\eval_s : \Tot^\circ(L_i) \rightarrow \C$, $\eval_s(l) := s(l^{\otimes m})$, be the associated function of (\cref{lemma::sectionRingsAreFunctions}). Choose generators $s_1, \ldots, s_{K_1}$ of $R_1$ and $s_{K_1 + 1}, \ldots, s_K$ of $R_2$ and let $\Psi_1:= s_1 \times \cdots \times s_{K_1}$, $\Psi_2 = s_{K+1} \times \cdots s_K$; $\Psi_i : C_i \rightarrow \C^{K_i}$, and put $\Psi:= \Psi_1 \times \Psi_2$. By the properties above, $\Psi_i$ is injective, an immersion, and misses $0$. 
		It is moreover proper: writing $v = r(v) \cdot u$ with $u \in \{ r = 1 \}$ we have the homogeneity equation $|\eval_{s_j}(v)| = r(v)^{d_j} |\eval_{s_j}(u)|$ ($d_j = \deg s_j \geq 1$). The function $u \mapsto \max_j |\eval_{s_j}(u)|$ is continuous and positive on the compact $\{ r = 1 \}$, so the homogeneity relation confines $r$ on the preimage of any compact subset of $\C^{K_i} \setminus \{ 0 \}$ to a compact interval. Hence $\Psi_i$ is a closed holomorphic embedding of $C_i$ onto a submanifold of $\C^{K_i} \setminus \{ 0 \}$. By Remmert's proper mapping theorem (\cite[Chapter~10,~\S6.1]{grauert_remmert::coherentAnalyticSheaves}), $\Psi_i(C)$ is an analytic set in $\C^{K_i} \setminus \{ 0 \}$. Thus $\Psi(C)$ is a closed analytic set in $(\C^{K_1} \setminus \{ 0 \}) \times (\C^{K_2} \setminus \{ 0 \} )$.

		\newProofStep
		\label{step::RIntegrallyClosed}
		We show $R = R_1 \otimes_\C R_2$ is integrally closed in its fraction field. \cref{step::weightDecompositionIntegral} gives an injection $R \hookrightarrow \holo(C_1 \times C_2)$, thus $R$ is a domain since $C_1 \times C_2$ is connected. $R$ is also finitely generated (\cref{proposition::sectionRingFinitelyGenerated}).
		
		Let $F$ lie in the fraction field of $R$ and satisfy
		\begin{equation}
			\label{eqn::equationFIntegralOverR}
			F^d + a_1 F^{d-1} + \cdots + a_d = 0, \qquad a_j \in R .
		\end{equation}
		Write $F = P/Q$ with $P, Q \in R$, $Q \not\equiv 0$, so $F$ is holomorphic off the nowhere dense analytic set $Z(Q) \subset C_1 \times C_2$. Bounding the roots of a polynomials by its coefficients gives $|F| \leq 1 + \sum_j |a_j|$, with continuous right-hand side; so $F$ is locally bounded near $Z(Q)$ and extends holomorphically across it by the Riemann extension theorem (\cite[Ch.~7,~\S1.3]{grauert_remmert::coherentAnalyticSheaves}), applied on the manifold $C_1 \times C_2$.
		
		Each $a_j$ is a finite sum of bihomogeneous pieces $a_{j,(\alpha,\beta)}$ (\cref{step::weightDecompositionIntegral}); let $D$ be the largest degree among them. Fix $v = (v_1, v_2)$ and expand $\lambda \mapsto F(\lambda \cdot v)$, holomorphic on $(\C^*)^2$ by the above, as the double Laurent series 
		$\sum_{(a,b)} F_{a,b}(v) \lambda_1^a \lambda_2^b$ of \cref{step::weightDecompositionIntegral}. From 
		\[
			a_j(\lambda \cdot v)| \leq \max(1, |\lambda_1|)^D \max(1, |\lambda_2|)^D \sum_{(\alpha, \beta)} |a_{j,(\alpha, \beta)}(v)|
		\]
		and the root bound, 
		$|F(\lambda \cdot v)| \leq C(v) \max(1, |\lambda_1|)^D \max(1, |\lambda_2|)^D$.
		The Cauchy estimate then gives, for all radii $\rho_1, \rho_2 > 0$,
		\begin{equation}
			\label{eqn::boundFab}
		|F_{a,b}(v)| \leq C(v) \, \rho_1^{-a} \max(1, \rho_1)^D \, \rho_2^{-b} \max(1, \rho_2)^D .
		\end{equation}
		Letting $\rho_i \to \infty$ in \eqref{eqn::boundFab} shows that $F_{a,b} \rightarrow 0$ when $a > D$ or $b > D$, and $\rho_i \to 0$ when $a < 0$ or $b < 0$. Hence $F = \sum_{0 \leq a, b \leq D} F_{a,b}$ with each $F_{a,b} \in R_1^a \otimes_\C R_2^b$ (\cref{step::weightDecompositionIntegral}), so $F \in R$.
		
		\newProofStep
		\label{step::analytificationIsNormal}
		The analytification $\Spec(R)^{\mathrm{an}}$ is a normal complex space. Indeed, finitely generated $\C$-algebras are excellent and the completion of an excellent normal local ring is normal (\cite[Thm.~23.9]{matsumura::commutativeRingTheory}; since $R$ is finitely generated (\cref{proposition::sectionRingFinitelyGenerated}), the completions of the local rings of $\Spec(R)$ are normal. But the algebraic and analytic local rings of $\Spec(R)$ have isomorphic completions by (\cite[{n\textsuperscript{o}~6,~Prop.~3}]{serre::GAGA}). Finally, $\holo_{\Spec(R)^{\mathrm{an}},x}\to \widehat{\holo}_{\Spec(R)^{\mathrm{an}},x}$ is faithfully flat and normality descends along a faithfully flat map. Thus every local ring of $(\Spec(R))^{\mathrm{an}}$ is integrally closed (which is equivalent to geometric normality - \cite{grauert_remmert::coherentAnalyticSheaves}).
		
		\newProofStep
		\label{step::imageIsProductSteinCompletion}
		Denote by $\widetilde C$ the Zariski closure of $\Psi(C)$ inside $\C^K$. We show now that
		$\widetilde C \simeq X$ i.e. that $\widetilde C$ is the product of the Stein completions of the open cones. 
		
		Let 
		$I := \ker(\C[T_1, \ldots, T_K] \rightarrow R)$. Then by construction $\widetilde C = V(I)$. The only difficulty is showing that the regular locus of $\widetilde C$ is not larger than $\Psi(C)$.
		
		By \cref{step::analytificationIsNormal}, $(\widetilde C)^{\mathrm{an}}$ is a normal complex space. Since it's a closed analytic subset of $\C^K$, it is Stein. Furthermore, as seen in \cref{step::embeddingViaGenerators}, $\Psi(C)$ is a closed analytic subset of $(\C^{K_1} \setminus \{ 0 \}) \times (\C^{K_2} \setminus \{ 0 \})$.
		We seek to apply the Identity Lemma (\cite[\S9.1.1]{grauert_remmert::coherentAnalyticSheaves} to the complex space $\widetilde C \setminus \Sing(\widetilde C)$ and the closed analytic subset $\Psi(C)$ of $\widetilde C \setminus \Sing(\widetilde C)$; its hypotheses are satisfied because $\widetilde C \setminus \Sing(\widetilde C)$ is connected by \cite[\S9.1.2]{grauert_remmert::coherentAnalyticSheaves} since 
		$\widetilde{C}^{\mathrm{an}}$ is irreducible. We conclude that
		$\widetilde{C} \setminus \Sing(\widetilde C) = \Psi(C)$.
		
		Moreover, retracing the same argument for the embedding of a single cone shows that $\widetilde{C_i}$ is the union of the embedding of $C_i$ in $\C^{K_i}$ with $0 \in \C^{K_i}$, which shows that $\widetilde C$ is indeed (biholomorphic to) the product of the Stein completions $X_1 \times X_2$.
		 
		 \newProofStep
		 	\label{step::actionOnEmebedding}
		 	It remains to observe what happens to the flows of the Reeb fields on $X$ when $X$ is seen as $\widetilde C$ from \cref{step::imageIsProductSteinCompletion}.
		 	For each flow, we will find appropriate contractions inside the group $G_\alpha$ defining the product $S_1 \times S_2$ as $(C_1 \times C_2)/G_\alpha$, such that the flow is in the Zariski closure of the contraction.
		 	
		 	As in \cref{step::weightDecomposition}, let $V:= R_i^m$ be a space in which a generator $s$ which we used to produce the embedding $\Psi$ lives. By the choice made in \cref{step::weightDecomposition}, there exists a character $\chi$ of $T^k_\C$ such that $s \in V_\chi$.
		 	
		 	Denote $\tilde\rho$ the action of $T^k_\C$ on $\Im\Psi$ given by
		 	$
		 		\tilde\rho(g) \Psi(l) :=
		 		\Psi (\rho(g) l)
		 	$
		 	
		 	By the properties of the exponential map and definitions:
		 	\begin{equation}
		 		\label{eqn::actionAndCharactersRelation}
		 			\rho(\exp(X))s = \exp(d_0 \chi (X)) s, \quad \forall X \in \LieAlg{t}^k_\C
		 	\end{equation}
		 	
		 	\cref{eqn::actionAndCharactersRelation} shows that on the image of the embedding via the generators $s_1, \ldots, s_K$ of $R$, we have:
		 	
		 	\begin{equation}
		 		\label{eqn::actionOnImageEmbedding}
			 	\rho(\exp(X)) \Psi(l) =
			 	(
			 		e^{d_0 \chi_1(X)} ev_{s_1}(l), 
			 		\ldots,
			 		e^{d_0 \chi_K(X)} ev_{s_K}(l)
			 	)
		 	\end{equation}
		 	where $\chi_j$ are characters such that $s_j \in (R_{i_j, m_j})_{\chi_j}$, $i_j \in \{ 1, 2 \}$, $m_j \in \N^{>0}$.
		 	
		  	Because $\exp: \LieAlg{t}^k_\C \rightarrow T^k_\C$ is surjective,
		 	\eqref{eqn::actionOnImageEmbedding} shows that on $\Im \Psi$ the action of $T^k_\C$ is a sub-action of the standard diagonal action of $(\C^*)^K$ on $\C^K$ (where $K$ has no connection with $k$).
		 	
		 	We see elements of $R_i^m$ as holomorphic functions on $C_i$ via \cref{lemma::sectionRingsAreFunctions}. 
		 	By \cref{step::actionOnCones}, the action map $T^{k_i}_\C \times C \rightarrow C$ is holomorphic, which entails that the action $\rho: T^{k_i}_\C \rightarrow \GL(V)$ is holomorphic, as it can be seen by writing $\rho$ using Cramer's rule for a basis of $V$, which consists of holomorphic functions. Therefore $\rho$ is in each variable an exponential map, which, together with the weight decomposition \cref{eqn::weightDecomposition}, readily shows that it is algebraic. Thus $\tilde\rho(T^k_\C)$ is an algebraic subgroup of $(\C^*)^K$.
		 	
		 	\newProofStep
				\label{step::completingReebsInZariskiArg}	 	
		 	Let now $\alpha \in \C \setminus \R$ defining $G_\alpha$ and the complex structure on $S_1 \times S_2$. For the remainder of the proof, our goal is for each $\xi_i$ and each $c \in \R$ to produce a contraction of $C$ which is a contraction on each $C_i$, belongs to $G_\alpha$, and contains $c \xi_i$ in its Zariski closure.
		 	
		 	We will fix $\gamma = \exp(X)$ later. For the moment, we observe that by \cref{step::actionOnEmebedding} the action such a $\gamma$ will be in any case algebraic, so it makes sense to speak of the Zariski closure $\mathcal{G}$ of $\gamma$ inside $(\C^*)^K$. By \cite[Section~16.2,~Proposition~p.~103]{humphreys::linearAlgGroups}, $\mathcal{G}$ is the intersection of the kernels of some algebraic characters of $(\C^*)^K$. 
		 	
		 	Therefore, it is enough to find a setup in which whenever $\chi$ is an algebraic character of $(\C^*)^K$ with the property that $\chi(\tilde\rho (\exp (X))) = 1$, it follows automatically that $\chi(\tilde\rho (\exp (c\xi_i))) = 1$.
		 	
		 	An algebraic character of a complex torus is a monomial with integer powers in the entries of the torus (\cite[p.~102]{humphreys::linearAlgGroups}), say $\chi = x_1^{n_1} \cdots x_K^{n_K}$, $n_j \in \Z$. 
		 	Thus, the previous  requirement reduces to showing that whenever equation:
		 	
		 	\begin{equation}
		 		\label{eqn::characterTrivialExplicit}
		 		\exp \left(
		 			\sum_{j=1}^K n_j d_0 \chi_j(Z) 
		 		\right) = 1
		 	\end{equation}
		 	holds for $Z=X$ defining the contraction $\gamma$, it also holds for $Z=c \xi_i$ for at least a constant $c$ (see \cref{step::choicesOfContractions} for why this suffices).
		 	
		 \newProofStep
		 	\label{step::fromContractionFieldToReeb}
		 Since $\chi_j$ are (holomorphic, hence algebraic) characters of $T^k_\C$, they are monomials with integer powers. Hence $\chi_j(T^k) \subset S^1$, where $T^k = (S^1)^k$ is the compact torus inside $T^k_\C$. Therefore, since $\xi_i \in T^k$, $d_0 \chi_j(\xi_i)$ is in the Lie algebra of $S^1$ whenever $d_0 \chi_j(\xi_i) \neq 0$, which means that
		 \begin{equation}
		 	\label{eqn::charactersOnReebFields}
		 	\left\lvert \exp(n_j d_0\chi_j(\xi_i)) \right\rvert = 1
		 \end{equation}
		 Because $\xi_i$ and $E_i$ commute, \eqref{eqn::charactersOnReebFields} establishes that if \eqref{eqn::characterTrivialExplicit} holds for $X_i = aE_i + b\xi_i$, then it also holds for $aE_i$. But then \eqref{eqn::characterTrivialExplicit} must also hold for $b\xi_i$.
		 
		 This argument shows that if \eqref{eqn::characterTrivialExplicit}
		 holds for a linear combination of the Reeb and Euler fields, then it holds for the part of the linear combination which only contains the Reeb fields. 
		 
		 \newProofStep
		 	\label{step::choicesOfContractions}
		 We show now that for each $c\xi_i$, $c \in \R$, there exists at least one contraction in $G_\alpha$ which is also a contraction on both factors $C_1$ and $C_2$ and such that $\exp(c\xi_i)$ sits in the Zariski closure of it.
		 
		 Recall that $G_\alpha = \{ (t, \alpha t) : t \in \C\} \subset \Aut(C)$.	 
		 Set
		 $$
		 	c_{\xi_1} (c, \alpha, r) := 
		 	\frac{1}{|\alpha|^2} \left(
		 		c \Re \alpha - r \Im \alpha
		 	\right)
		 ,\quad
		 c_{\xi_2} (c, \alpha, r) := 
			 c \Re \alpha + r \Im \alpha
		 $$
		 
		 For $c\xi_1$, let $t:= r_1 + \sqrt{-1}c$ and
		 $$
		 X_1 :=
			 r_1 E_1 
			 + c \xi_1
			 + \Re(\alpha t ) E_2 
			 + c_{\xi_2} (c, \alpha, r)\xi_2
		 $$
		 
		 By \cref{step::fromContractionFieldToReeb}, if \eqref{eqn::characterTrivialExplicit} is satisfied for $X$, then it's also satisfied for $c \xi_1 + c_{\xi_2} (c, \alpha, r) \xi_2$.
		 Define
		 $
		 \Lambda_i := \spanSimple_\Z (d_0 \chi_1(\xi_i), \ldots, d_0 \chi_K (\xi_i))
		 , \quad i=\overline{1,2}
		 $
		 
		 If there exists a suitable $r_1$ such that 
		 for any $0 \neq W_2 \in \Lambda_2$ and any
		 $W_1 \in \Lambda_1$ we have $c W_1 + c_{\xi_2}(c, \alpha, r_1) W_2 \not\in 2 \pi \Z$, then it is guaranteed by \cref{step::fromContractionFieldToReeb} that whenever \eqref{eqn::characterTrivialExplicit} holds for $X_1$, it also holds for $c \xi_1$, as desired.
		 
		 Now observe that the condition $c W_1 + c_{\xi_2}(c, \alpha, r_1) W_2 \not\in 2 \pi \Z$ translates to
		 $$
		 	r_1 \not\in \frac{2 \pi \Z - c W_1 - c \Re(\alpha) W_2}{\Im(\alpha) W_2}
		 $$
		 Since $\Lambda_1$ and $\Lambda_2$ are countable, it suffices to show that we always have uncountably many choices for $r_1$.
		 
		 For $c\xi_2$, we reparametrise $G_\alpha$ as $\{ (v/\alpha, v) : v \in \C \}$. Then we look for a suitable $v=r_2 + \sqrt{-1}c$ and a vector field $X_2$ of the form:
		 $$
		 X_2 :=
			 \Re(v/\alpha) E_1 
			 + c_{\xi_1}(c, \alpha, r_2) \xi_1
			 + r_2 E_2 
			 + c\xi_2
		 $$
		 such that $\exp(X_2)$ is a contraction on both factors.
		 Similarly to the situation for $c\xi_1$, if there exists $r_2$ such that for any
		 $0 \neq W_1 \in \Lambda_1$ and any $W_2 \in \Lambda_2$ we have 
		 $
		 c_{\xi_1}(c, \alpha, r_2) W_1 + c W_2 \not\in 2 \pi \Z
		 $
		 then \cref{step::fromContractionFieldToReeb} guarantees that whenever \eqref{eqn::characterTrivialExplicit} holds for $X_2$, it also holds for $c \xi_2$, as desired. But the condition
		 $
		 c_{\xi_1}(c, \alpha, r_2) W_1 + c W_2 \not\in 2 \pi \Z
		 $
		 translates to
		 $$
		 r_2 \not\in \frac{
		 		c \Re(\alpha) W_1 + 
		 		|\alpha|^2 \left(
		 		c W_2 - 2 \pi \Z
		 		\right) 
		 	}
		 	{
		 		\Im(\alpha) W_1
	 		}
		 $$
		 and since $\Lambda_1$ and $\Lambda_2$ are countable, it suffices to show that we always have uncountably many choices for $r_2$.
		 
		 To summarise, we are left to investigate, for a given $\alpha \in \C \setminus \R$ and a given $c \in \R$, when do we have an uncountable set of choices of $r_1 \in \R$ such that $\exp(X_1)$ is a contraction on both factors, as well as when we have an uncountable set of choices of $r_2 \in \R$ such that $\exp(X_2)$ is a contraction on both factors.
		 
		 \textbf{Case 1.} $\Re(\alpha) > 0$.
		 
		 Any choice of
		 $
		 r_1 < \min
		 \left(
		 0, \frac
		 	{
		 		c\Im(\alpha)
		 	}
		 	{
		 		\Re(\alpha)
		 	}
		 \right)
		 $ 
		 will guarantee that $\exp(X_1)$ is a contraction on both factors.  
		 Any choice of
		 $
		 r_2 < \min
		 \left(
		 	0, \frac
		 		{
		 			-\Im(v)\Im(\alpha)
		 		}
		 		{
		 			\Re(\alpha)
 				}
		 \right)
		 $
		 guarantees that $\exp(X_2)$ is a contraction on both factors.
		 
		 \textbf{Case 2.} $\Re(\alpha) \leq 0$.
		 
		  Denote by 
		  $S_1(\alpha) := \{ \Im(t): t \in \C, \Re(t) < 0, \Re(t\alpha) < 0 \}$.
		  In other words, $S_1(\alpha)$ gathers the imaginary parts of all the $C_1$-components of contractions in $G_\alpha$.
		 
		 The contraction conditions imply via similar choices as in case 1 that 
		 $$
		 S_1(\alpha) = 
			 \begin{cases}
			 	(0, \infty), \quad \Im(\alpha) > 0 \\
			 	(-\infty, 0), \quad \Im(\alpha) < 0
			 \end{cases}
		 $$
		 But if $Z$ satisfies \eqref{eqn::characterTrivialExplicit}, then so does $2\pi l Z$, for any $l \in \Z$. Therefore, $c\xi_1$ sits in the Zariski closure of at least one contraction for any given $c$. 
		 
		 For $c\xi_2$, as before we reparametrise $G_\alpha$ as $(v/\alpha, v)$, $v \in \C$.
		 Denote 
		 $S_2(\alpha) := \{ \Im(v): v \in \C, \Re(v) < 0, \Re(v/\alpha) < 0 \}$.
		 Then
		 $$
		 S_2(\alpha) = 
		 \begin{cases}
		 	(-\infty, 0), \quad \Im(\alpha) > 0 \\
		 	(0, \infty), \quad \Im(\alpha) < 0
		 \end{cases}
		 $$
		 Again by the observation about solutions of \eqref{eqn::characterTrivialExplicit}, $c\xi_2$ sits in the Zariski closure of at least one contraction for any given $c$. 
		 This completes the proof.
	\end{proof}
\end{theorem}

\hfill

{\small
	
	\noindent {\sc Vlad Marchidanu\\
		University of Bucharest, Faculty of Mathematics and Informatics, \\
            14 Academiei str., 70109 Bucharest, Romania\\
            also: \\
            Institute of Mathematics “Simion Stoilow” of the Romanian Academy \\
            21, Calea Grivitei Str. 010702-Bucharest, Romania \\
		\tt marchidanuvlad@gmail.com}
}
\end{document}